\documentclass[a4paper, 12pt]{amsart}
\usepackage{amssymb,amsmath,amsthm,a4}
\numberwithin{equation}{section}
\usepackage{enumerate,color}

\newtheorem{thm}{Theorem}[section]
\newtheorem{lem}[thm]{Lemma}

\newtheorem{prop}[thm]{Proposition}
\theoremstyle{definition}

\newcommand{\R}{\mathbb{R}}
\newcommand{\C}{\mathbb{C}}

\newcommand{\HP}{\mathbb{H}}
\newcommand{\A}{\mathcal{A}}

\begin{document}
\title[Composition operators on weighted Bergman spaces]{Composition operators on weighted Bergman spaces of a half plane}
\author[S.~J.~Elliott and A.~Wynn]{Sam Elliott and Andrew Wynn}

\address{Department of Pure Mathematics\\
  University of Leeds\\
  Leeds\\
  LS2~9JT\\
  UK}

\email{samuel@maths.leeds.ac.uk}

\address{Department of Mathematics\\
  University College London\\
  Gower Street\\
  London\\
  WC1E 6BT\\
  UK}

\email{andrew.wynn@ucl.ac.uk}

\subjclass[2000]{47B33}

\thanks{The authors would like to thank the staff at the ICMS (Maxwell Institute) in Edinburgh, and the organisers of the Workshop on Operator Theory and its Applications, held at the ICMS, where this work was initiated}
\date{\today}

\begin{abstract}
We use induction and interpolation techniques to prove that a composition operator induced by a map $\phi$ is bounded on the weighted Bergman space $\A^2_\alpha(\mathbb{H})$ of the right half-plane if and only if $\phi$ fixes $\infty$ non-tangentially, and has a finite angular derivative $\lambda$ there. We further prove that in this case the norm, essential norm, and spectral radius of the operator are all equal, and given by $\lambda^{(2+\alpha)/2}$.
\end{abstract}
\maketitle

\section{Introduction}
Analytic composition operators have been studied in a number of contexts, primarily on spaces of functions in the unit disc of the the complex plane. It has long been known as a consequence of the Littlewood subordination principle that all such operators are bounded on all the Hardy spaces, as well as a large class of other spaces of functions.

On the half-plane, however, things are somewhat more complicated. It is well know that there are unbounded composition operators on the half-plane. Indeed, in \cite{Matache2}, Valentin Matache proved that a composition operator is bounded on the Hardy space $H^2$ of the half plane if and only if the inducing map fixes the point at infinity, and has a finite angular derivative $\lambda$ there. Later, in \cite{ElliottJury} the first named author and Michael Jury sharpened this result, and showed that in the case when such a composition operator is bounded, the norm, essential norm and spectral radius of the operator are all equal to $\sqrt{\lambda}$. This in particular strengthened a result on non-compactness of composition operators produced by Matache in \cite{Matache}.

Noting that the Hardy space is effectively the Bergman space with weight $\alpha=-1$, we will take the known situation as a base case, and use induction and interpolation techniques to extend the results to all weighted Bergman spaces. In particular, we provide a formula for the norm which agrees with the known results for the Hardy space case. For a thorough discussion of Bergman spaces and their composition operators, see \cite{CowenMacCluer} or \cite{HedenmalmHaakanZhu}.

\section{Preliminaries}
Let $\HP$ denote the right half-plane $\{\Re z >0\}$. For $\alpha>-1$, the weighted Bergman space $\mathcal{A}^2_\alpha(\HP)$ contains those analytic functions $F:\HP\rightarrow\C$ for which
\[
\|F\|^2_{\A^2_\alpha(\HP)}:=\frac{1}{\pi}\int_{-\infty}^\infty\int_0^\infty x^\alpha|F(x+iy)|^2dxdy < \infty.
\]
For each $\alpha>-1$, the functions $\{k_\omega^\alpha;\omega\in\HP\}$ defined by
\begin{equation}
 k_\omega^\alpha(z):=\frac{2^\alpha(1+\alpha)}{(\overline{\omega}+z)^{2+\alpha}}, \qquad z \in \mathbb{H},
\end{equation}
are the reproducing kernels for $\A^2_\alpha(\HP)$. As such, they have the property that
\begin{equation}\label{KernelProperty}
 \left<f,k_\omega\right>_{\A^2_\alpha(\mathbb{H})}=f(\omega), \quad f\in\A_\alpha^2(\HP),{ } \omega\in\HP.
\end{equation}
In order to prove our result, we will show that a certain kernel is positive. We say a kernel $K(z,w)$ on $\mathbb H\times \mathbb H$ is positive if
\begin{equation*}
\sum_{i,j=1}^n c_i \overline{c_j} K(x_i, x_j) \geq 0
\end{equation*} 
for all $n\geq 1$, and all scalars $c_1, \dots c_n\in\C$ and points $x_1, \dots x_n\in \mathbb H$.

\begin{prop}[Nevanlinna] \label{Nevanlinna} A holomorphic function $\psi$ in
$\mathbb H$ has positive real part if and only if the kernel
\begin{equation*}
\frac{\psi(z)+\overline{\psi(w)}}{z+\overline{w}}
\end{equation*}
is positive.
\end{prop}

A sequence of points $z_n=x_n+iy_n$ in $\mathbb H$ is said to tend non-tangentially to $\infty$ if
\begin{enumerate}
 \item $x_n\to \infty$,
 \item the ratio $|y_n|/x_n$ is uniformly bounded.
\end{enumerate}
 We then say that a map $\phi:\mathbb H\to \mathbb H$ fixes infinity non-tangentially, and write $\phi(\infty)=\infty$, if $\phi(z_n)\to \infty$ whenever $z_n\to \infty$ non-tangentially.  If $\phi(\infty)=\infty$, we say that $\phi$ has a finite angular derivative at $\infty$ if the non-tangential limit
\begin{equation}\label{E:angdivH}
\lim_{z\to \infty} \frac{z}{\phi(z)}
\end{equation}
exists and is finite, under these circumstances, we write $\phi^\prime(\infty)$ for this quantity.

If we let $\psi$ be the self-map of $\mathbb D$ equivalent to $\phi$ via the standard M\"obius identification of the disc with the half-plane given by $\tau(\zeta)=\frac{1+\zeta}{1-\zeta}$, that is ${\psi = \tau^{-1} \circ\phi\circ\tau}$, then (\ref{E:angdivH}) is equal, by the Julia-Carath\'eodory theorem, the non-tangential limit of $\psi^\prime(\zeta)$ as $\zeta\to 1$, which is where the terminology comes from. Indeed, we have the following half-plane version of the Julia-Carath\'eodory theorem, proved in \cite{ElliottJury}.
\begin{lem}[Half plane Julia-Carath\'eodory theorem]  \label{AngularDeriv}
Let $\phi:\mathbb H\to \mathbb H$ be holomorphic.  The following are equivalent:
\begin{enumerate}
\item $\phi(\infty)=\infty$ and $\phi^\prime(\infty)$ exists;  
\item $\displaystyle{\sup_{z\in\mathbb H} \frac{\Re z}{\Re \phi(z)}<\infty}$;
\item $\displaystyle{\limsup_{z\to \infty} \frac{\Re z}{\Re \phi(z)}<\infty}$.
\end{enumerate}
Moreover the quantities in (2) and (3) are both equal to the angular derivative $\phi^\prime(\infty)$.
\end{lem}
\begin{lem}\label{KernelSums}
 Suppose that $K(\omega,z)$ is a positive kernel on $\mathbb{H} \times \mathbb{H}$ and let $c \geq 0$ be a positive constant. Then $\tilde K(\omega,z):= K(\omega,z)+ c$ is a positive kernel on $\mathbb{H} \times \mathbb{H}$.
\end{lem}
\begin{proof}
Since the analytic function $\psi(z)=z$ on $\mathbb{H}$ has positive real part, Proposition \ref{Nevanlinna} implies that $L(\omega,z) \equiv 1$ is positive. Since $\tilde K = K + cL$, it follows that $\tilde K$ is positive. 
\end{proof}
\section{Main Results}
For a natural number $n \geq  1$ and a holomorphic function $\phi:\mathbb{H} \rightarrow \mathbb{H}$ with finite angular derivative $\lambda$ at infinity, we define the kernel $K^n(\omega,z)$ on $\mathbb{H} \times \mathbb{H}$ by
\begin{equation*}
K^n(\omega, z):= \frac{ ( \phi(z) + \overline{\phi(\omega)})^n - \lambda^{-n} (z + \bar \omega)^n}{(z + \bar \omega)^n}, \qquad \omega, z \in \mathbb{H}.
\end{equation*}
\begin{lem} \label{KerPos}
Suppose that $\phi:\mathbb{H} \rightarrow \mathbb{H}$ has finite angular derivative $0 < \lambda < \infty$ at infinity. Then for every natural number $n \geq 0$, the kernel $K^{2^n}$ is positive.
\end{lem}
\begin{proof}
It is shown in \cite{ElliottJury} that $K^{1}$ is positive. Now suppose that $K^{2^n}$ is positive for some natural number $n \geq 1$. Then, using the fact that the numerator of $K^{2^{n+1}}$ is the difference of two squares,
\begin{align*}
K^{2^{n+1}}(& \omega,z)  =  \frac{ \left( ( \phi(z) + \overline{\phi(\omega)})^{2^{n}} \right)^2 - \left( \lambda^{-2^{n}}(z + \bar \omega)^{2^{n}}  \right)^2}{ (z + \bar \omega)^{2^{n+1}}}\\
& =  \frac{ (\phi(z) + \overline{\phi(\omega)})^{2^{n}} - \lambda^{-2^{n}}(z + \bar \omega)^{2^{n}} }{ (z + \bar \omega)^{2^{n}}} \cdot \frac{ (\phi(z) + \overline{\phi(\omega)})^{2^n} + \lambda^{-2^n}(z + \bar \omega)^{2^n} }{ (z + \bar \omega)^{2^n}} \\
& =  K^{2^n}(\omega,z) \cdot \left( \frac{ (\phi(z) + \overline{\phi(\omega)})^{2^n}}{ (z + \bar \omega)^{2^n}} +  \lambda^{-2^n} \right)\\
& =  K^{2^{n}}(\omega,z) \left( K^{2^{n}}(\omega,z) + 2 \cdot \lambda^{-2^n} \right).
\end{align*}
By assumption that $K^{2^n}$ is positive and Lemma \ref{KernelSums}, this is the product of two positive kernels, and hence, $K^{2^{n+1}}$ is positive by the Schur product theorem \cite{AglerMccarthy}. The result now follows by induction.
\end{proof}
As a result of Lemma \ref{KerPos}, it is possible to provide conditions for boundedness of composition operators on weighted Bergman spaces, for certain integer weights. 

\begin{prop}\label{Prop:2n-2}
Let $\phi : \mathbb{H} \rightarrow \mathbb{H}$ be holomorphic and let $n \geq 1$ be a natural number. The composition operator $C_\phi : \mathcal{A}_{2^n-2}^2(\mathbb{H}) \rightarrow \mathcal{A}_{2^n-2}^2(\mathbb{H})$ is bounded if and only if $\phi$ has finite angular derivative $0 < \lambda < \infty$ at infinity, in which case $\|C_\phi\| = \lambda^{2^{n-1}}$. 
\end{prop}
\begin{proof}
Let $n \geq 1$ be a natural number and define $\alpha:=2^n-2$. Following \cite{ElliottJury}, if it can be shown that 
\begin{equation} \label{ShowPos}
\lambda^{2^n} \langle k_\omega^\alpha, k_z^\alpha \rangle_{\mathcal{A}_\alpha^2(\mathbb{H})} - \langle C_\phi^\ast k_\omega^\alpha, C_\phi^\ast k_z^\alpha \rangle_{\mathcal{A}_\alpha^2(\mathbb{H})}
\end{equation}
is a positive kernel, then $C_\phi : \mathcal{A}_\alpha^2(\mathbb{H}) \rightarrow \mathcal{A}_\alpha^2(\mathbb{H})$ is bounded with $\|C_\phi\| \leq \lambda^{2^{n-1}}$. Using the fact that $C_\phi^\ast k_\omega^\alpha = k^\alpha_{\phi(\omega)}$ and (\ref{KernelProperty}), it follows that (\ref{ShowPos}) is equal to 
\begin{equation*}
2^\alpha(1+\alpha) \left(\frac{\lambda^{2^n}}{(z+\bar \omega)^{2^n}} - \frac{1}{(\phi(z) + \overline{\phi(\omega)})^{2^n}}\right). 
\end{equation*}
This can easily be seen to factorise as
\begin{equation*}
 \lambda^{2^n} \frac{ 2^\alpha(1+\alpha)}{(\phi(z) + \overline{\phi(\omega)})^{2^n}} \cdot \frac{ (\phi(z) + \overline{\phi(\omega)})^{2^n} - \lambda^{-2^n} (z + \bar \omega)^{2^n}}{ (z + \bar \omega )^{2^n}},
\end{equation*}
which is just
\begin{equation*}
\lambda^{2^n} \langle k_{\phi(\omega)}^\alpha , k_{\phi(z)}^\alpha \rangle_{\mathcal{A}_\alpha^2(\mathbb{H})} \cdot K^{2^n}(\omega,z).
\end{equation*}
This is positive, being the product of positive kernels and positive scalars.

For the converse, the calculation is similar to the Hardy space case. If the composition operator $C_\phi : \mathcal{A}_\alpha^2(\mathbb{H}) \rightarrow \mathcal{A}_\alpha^2(\mathbb{H})$ is bounded and $\|C_\phi\| \leq M$ then, 
\begin{align*}
\frac{2^\alpha(1+\alpha)}{2^{2+\alpha}( \Re \phi(z))^{2+\alpha} }  
 =  \|k_{\phi(z)}^\alpha \|_{\mathcal{A}_\alpha^2(\mathbb{H})}^2 
& =  \|C_\phi^\ast k_z^\alpha \|_{\mathcal{A}_\alpha^2(\mathbb{H})}^2 \\
& \leq  M^2 \|k_z^\alpha\|_{\mathcal{A}_\alpha^2(\mathbb{H})}^2 \\
& =  M^2 \frac{2^\alpha(1+\alpha)}{2^{2+\alpha} (\Re z)^{2+\alpha}}. 
\end{align*}
As such,
\[
\frac{\Re(z)}{\Re(\phi(z))} \le M^{2/(2+\alpha)},
\]
hence by Lemma \ref{AngularDeriv}, $\phi$ has finite angular derivative 
\[
\phi'(\infty) = \lambda \leq \|C_\phi\|^{2/(2+\alpha)} = \|C_\phi\|^{2^{-(n-1)}}.
\]
By the first part of the proof, the norm of $C_\phi$ must be at most $\lambda^{2^{n-1}}$, and by the second part it must be at least that large. It follows that indeed
\[
\|C_\phi\| = \lambda^{2^{n-1}}.
\]
\end{proof}

Proposition \ref{Prop:2n-2} tells us that the result holds for particular integral values of $\alpha$ of arbitrarily large size. We proceed by interpolating for the spaces $\A_\alpha^2(\mathbb{H})$, where $2^n<\alpha<2^{n+1}$. The following weighted version of the Paley-Wiener Theorem (see \cite{Duren} or \cite{GuallardoPartingtonSegura}) will be useful.
\begin{lem} \label{PaleyWiener}
The Bergman space $\A_\alpha^2(\mathbb{H})$ is isometrically isomorphic, via the Laplace transform $\mathcal{L}$, to the space $L^2(\R_+,d\mu_\alpha)$. Here, 
\[
 d\mu_\alpha = \frac{\Gamma(1+\alpha)}{2^\alpha t^{\alpha+1}} dt,
\]
and $dt$ is Lebesgue measure on $\mathbb{R}_+:=(0,\infty)$. 
\end{lem}

\begin{thm}\label{UpperBound}
Let $\phi : \mathbb{H} \rightarrow \mathbb{H}$ be holomorphic and $\alpha>-1$. The composition operator $C_\phi : \mathcal{A}_\alpha^2(\mathbb{H}) \rightarrow \mathcal{A}_\alpha^2(\mathbb{H})$ is bounded if and only if $\phi$ has finite angular derivative $0 < \lambda < \infty$ at infinity, in which case $\|C_\phi\| = \lambda^{(2+\alpha)/2}$. 
\end{thm}
\begin{proof}
Let $\alpha > -1$. By Proposition \ref{Prop:2n-2}, the result holds if $\alpha$ is of the form  $\alpha=2^n-2$. Hence, it may be assumed without loss of generality that there exists a natural number $n \geq 0$ such that $\alpha \in (2^n - 2, 2^{n+1}-2)$. Write $A:=2^n-2$, $B:=2^{n+1} -2$. In the following, for simplicity, write $L^2(d\mu)$ for $L^2(\mathbb{R}_+,d\mu)$. Define a linear operator 
\begin{equation*}
T:L^2(d\mu_A) \rightarrow L^2(d\mu_A); \qquad T:L^2(d\mu_B) \rightarrow L^2(d\mu_B)
\end{equation*}
by $T := \mathcal{L}^{-1} \circ C_\phi \circ \mathcal{L}$. Since $\mathcal{L}$ is an isometric isomorphism between the respective spaces (Lemma \ref{PaleyWiener}), Proposition \ref{Prop:2n-2} implies that 
\begin{align*}
\|T\|_{L^2(d\mu_A) \rightarrow L^2(d\mu_A)} &= \|C_\phi\|_{\mathcal{A}_A^2(\mathbb{H}) \rightarrow \mathcal{A}_A^2(\mathbb{H})} = \lambda^{2^{n-1}} = \lambda^{(2+A)/2};\\
\|T\|_{L^2(d\mu_B) \rightarrow L^2(d\mu_B)} &= \|C_\phi\|_{\mathcal{A}_B^2(\mathbb{H}) \rightarrow \mathcal{A}_B^2(\mathbb{H})} = \lambda^{2^n} = \lambda^{(2+B)/2}.
\end{align*}
(Note that in the case $n=0$, $\mathcal{A}^2_A(\mathbb{H})$ should be replaced by the Hardy space $H^2(\mathbb{H})$). 
Since $\alpha \in (A,B)$, there exists $\theta \in (0,1)$ such that $\alpha = A(1-\theta) + B \theta$. By the Stein-Weiss interpolation theorem \cite[Corollary 5.5.4]{BerghLofstrom},
\begin{equation} \label{Interp}
\|T\|_{L^2(dw) \rightarrow L^2(dw)} \leq \lambda^{(2+A)(1-\theta)/2} \lambda^{(2+B)\theta/2} = \lambda^{(2+\alpha)/2},
\end{equation}
where
\begin{equation*}
dw = \frac{ \Gamma(1+A)^{1-\theta} \Gamma(1+B)^\theta}{2^{A(1-\theta)+ B\theta} t^{A(1-\theta)+B\theta +1}} dt = \frac{ \Gamma(1+A)^{1-\theta} \Gamma(1+B)^\theta}{2^\alpha t^{1+\alpha}} dt
\end{equation*}
By Lemma \ref{PaleyWiener}, for any $g \in \mathcal{A}_\alpha^2(\mathbb{H})$ there exists $f \in L^2(d\mu_\alpha)$ such that $\mathcal{L}f =g $ and $\|g\|_{\mathcal{A}_\alpha^2(\mathbb{H})} = \|f\|_{L^2(d\mu_\alpha)}$. Thus,
\begin{align*}
\|C_\phi g\|_{\mathcal{A}_\alpha^2(\mathbb{H})}  =  \|C_\phi (\mathcal{L}f)\|_{\mathcal{A}_\alpha^2(\mathbb{H})} & =  \| \mathcal{L}( T f) \|_{\mathcal{A}_\alpha^2(\mathbb{H})}\\
& =  \|Tf\|_{L^2(d\mu_\alpha)}\\
& =  \frac{\Gamma(1+\alpha)^{1/2}}{\Gamma(1+A)^{(1-\theta)/2} \Gamma(1+B)^{\theta/2}} \|Tf\|_{L^2(dw)}\\
\text{(by (\ref{Interp}))}  & \leq   \frac{\lambda^{(2+\alpha)/2}  \Gamma(1+\alpha)^{1/2}}{\Gamma(1+A)^{(1-\theta)/2} \Gamma(1+B)^{\theta/2}} \|f\|_{L^2(dw)}\\
& =  \lambda^{(2+\alpha)/2}  \|f\|_{L^2(d\mu_\alpha)}\\
& =  \lambda^{(2+\alpha)/2}  \|g\|_{\mathcal{A}_\alpha^2(\mathbb{H})}.
\end{align*}
As such, $C_\phi$ is bounded with $\|C_\phi\|\le\lambda^{(2+\alpha)/2}$.

For the converse assume that $C_\phi$ is bounded. Then by exactly the same proof as the second half of Proposition \ref{Prop:2n-2}, it follows that $\phi$ has finite angular derivative $\lambda$ and that $\|C_\phi\| \geq \lambda^{(2+\alpha)/2}$.
\end{proof}

The following results, concerning the spectral radius and essential norm of $C_\phi$, can be deduced from Theorem \ref{UpperBound} by the methods used in \cite{ElliottJury} for the Hardy space $H^2(\mathbb{H})$. 

\begin{thm}
 If $C_\phi$ is bounded on $\A_\alpha^2(\mathbb{H})$, then its spectral radius and norm are equal.
\end{thm}

\begin{thm}  Every bounded composition operator on $\A_\alpha^2(\mathbb H)$ has essential norm equal to its operator norm.  In particular, since the zero operator is not a composition operator, there are no compact composition operators on any of the spaces $\A_\alpha^2(\mathbb H)$.
\end{thm}

\section*{Acknowledgements}
The authors would like to thank both Jonathan Partington and Michael Jury for helpful mathematical conversations which have aided the progress of this work.

\end{document}